\documentclass[12pt]{amsart}
\usepackage{graphicx,verbatim,amssymb}
\usepackage[active]{srcltx}

\newtheorem{thm}{Theorem}[section]

\newtheorem{lem}[thm]{Lemma}
\newtheorem{prop}[thm]{Proposition}
\theoremstyle{definition}

\theoremstyle{remark}

\numberwithin{equation}{section}

\newcommand{\R}{\mathbb R}
\def\RP{\mathbb{R}P}
\def\CP{\mathbb{C}P}
\def\C{\mathbb{C}}
\def\Lc{\mathcal{L}}
\def\Uc{\mathcal{U}}
\newcommand{\eps}{\varepsilon}

\def\d{\partial}

\def\0{\varnothing}
\newcommand\pd[2]{\frac{\partial #1}{\partial #2}}

\begin{document}

\title[Planarizations]{Planarizations and maps taking lines to linear
webs of conics}
\author{Vladlen Timorin}

\address{Faculty of Mathematics and Laboratory of Algebraic Geometry\\
National Research University Higher School of Economics\\
7 Vavilova St 112312 Moscow, Russia}

\address[Vladlen~Timorin]
{Independent University of Moscow\\
Bolshoy Vlasyevskiy Pereulok 11, 119002 Moscow, Russia}

\email{vtimorin@hse.ru}

\thanks{
Partially supported by
the Deligne fellowship, the Simons-IUM fellowship, RFBR grants 10-01-00739-a,
11-01-00654-a, MESRF grant MK-2790.2011.1, and
AG Laboratory NRU-HSE, MESRF grant ag. 11 11.G34.31.0023
}


\begin{abstract}
Aiming at a generalization of a classical theorem of M\"obius,
we study maps that take line intervals to plane curves, and
also maps that take line intervals to conics from certain linear systems.
\end{abstract}
\maketitle
\section{Introduction}

\subsection{The problem}
In 1827, M\"obius \cite{M} proved that a continuous one-to-one map
$\RP^2\to\RP^2$ that takes any straight line to a straight line
is a projective transformation.
As was noted later by von Staudt \cite{vS}, the continuity assumption was
superfluous.

M\"obius used the term {\em collineation} for a map that takes
collinear points to collinear points.
More precisely, let $U\subset\RP^2$ be an open subset.
A map $F:U\to\RP^n$ is called a {\em collineation} if
there exists an open set of lines $L\subset\RP^2$ such that
$L\cap U\ne\0$,
and the set $F(U\cap L)$ is collinear,
i.e. lies in a projective line.
A very minor modification of M\"obius' original argument yields the following theorem.
{\em Any continuous collineation $F:U\to\RP^n$ is a restriction of a
projective embedding, or a map to a subset of a projective line,
possibly after replacing $U$ with a smaller open set.}

By analogy, we define a {\em planarization} as a
map $F:U\to\RP^n$, for which there exists an open set of lines $L\subset\RP^2$
such that $U\cap L\ne\0$, and the set $F(U\cap L)$ lies in a projective hyperplane.
We consider the following problem:
{\em describe all planarizations $F:U\to\RP^n$}.
This problem is a direct generalization of the problem considered by M\"obius.

On the other hand, it is motivated by the following class of problems.
Let $\Lc$ be a linear system of algebraic curves in $\RP^2$, and
$U\subset\RP^2$ an open set.
We say that a map $f:U\to\RP^2$ {\em takes all lines to $\Lc$-curves}
if $f(U\cap L)$ belongs to a curve from $\Lc$ for every line $L\subset\RP^{2}$.
For various linear systems $\Lc$, we would like to know all sufficiently
smooth maps that take lines to $\Lc$-curves.
In the case, where $\Lc$ consists of all circles, this problem came
from Nomography (see e.g. \cite{GKh}); it was solved by A. Khovanskii \cite{Kh}.
Maps taking lines to circles have also been investigated in higher dimensions
\cite{Iz,rect,cliff}.
Large classes of such maps are provided by generalized Hopf fibrations \cite{maplc}
and, more generally, by quadratic rational parameterizations of spheres
\cite{circquad}.

If a sufficiently smooth map $f:U\to\RP^2$ takes lines to $\Lc$-curves,
then it gives rise to a planarization $F=\Phi_{\Lc}\circ f:U\dashrightarrow\RP^n$,
where $\Phi_\Lc:\RP^2\dashrightarrow\RP^n$ is the rational map, whose homogeneous
components are generating equations for $\Lc$ (more details are given below).
This was my original motivation for studying planarizations.

\subsection{Examples of planarizations}
Let us now give some examples of planarizations.

{\em Trivial planarizations.} A trivial planarization is a map
$F:U\to\RP^n$, whose image lies in some hyperplane.

{\em Co-trivial planarizations.}
A planarization $F:U\to\RP^n$ is {\em co-trivial} if there is a point
$o\in\RP^n$ such that, for some open set of lines $L\subset \RP^2$,
the set $F(U\cap L)$ is contained in a hyperplane passing through $o$.
An example of a co-trivial planarization can be constructed as follows.
Consider any planarization $G:U\to\R^{n-1}$, e.g. a projective map.
Also consider any continuous function $\phi:U\to\R$.
Then we can define the map $F:U\to\R^n=\R^{n-1}\times\R$
by the formula $F(u)=(G(u),\phi(u))$.
This is a co-trivial planarization
(the corresponding point $o$ is $(0,\infty)$, i.e. the intersection of
the projective closure of $0\times\R$ with the hyperplane at infinity).

{\em Rational degree $d$ maps.}
A rational degree $d$ map is a rational map $\RP^2\dashrightarrow\RP^n$
given in homogeneous coordinates by homogeneous degree $d$ polynomials without
a nonconstant common polynomial factor.
A rational degree $d$ map can have some {\em points of indeterminacy}, where
all defining polynomials vanish simultaneously.
Any rational degree $d$ map $\RP^2\dashrightarrow\RP^n$  with $n>d$ is a planarization.
To prove this statement, it suffices to show that the image of
every rational degree $d$ map $\RP^1\dashrightarrow\RP^n$ is contained in a
$d$-dimensional projective subspace.
Let $[u_0:u_1]$ be homogeneous coordinates in $\RP^1$.
The lift of a rational degree $d$ map $Q:\RP^1\dashrightarrow\RP^n$ to $\R^{n+1}$
has the form
$$
\tilde Q(u_0,u_1)=u_0^d A_0 + u_0^{d-1}u_1A_1 +\dots + u_1^dA_d,
$$
where $A_0$, $\dots$, $A_d$ are some constant vectors in $\R^{n+1}$.
Hence the image of $Q$ is contained in the projectivization of
the subspace spanned by $A_0$, $\dots$, $A_d$.
This projective subspace is at most $d$-dimensional.

{\em Duality.}
Projective duality between points and hyperplanes
yields a remarkable duality on planarizations.
Recall that points $\hat x$ in $\RP^{n*}$ correspond to hyperplanes
$P_{\hat x}$ in $\RP^n$.
In particular, points $\hat a\in\RP^{2*}$ correspond to lines
$L_{\hat a}\subset\RP^2$, and points $a\in\RP^2$ correspond to lines
$L_{a}\subset\RP^{2*}$.
Let $U\subset\RP^2$ be an open set, and $F:U\to\RP^n$ a planarization.
Then we can define the dual planarization $\hat F:\hat U\to\RP^{n*}$ as follows.
By definition, the set $\hat U\subset\RP^{2*}$ is the set of all
points $\hat a\in\RP^{2*}$ such that
$F(U\cap L_{\hat a})$ lies in a unique hyperplane $P_{\hat x}$.
We set $\hat F(\hat a)=\hat x$.
Thus we defined a map $\hat F:\hat U\to\RP^{n*}$.
We need to prove that this map is a planarization: for every $a\in U$,
the set $\hat F(\hat U\cap L_a)$ is contained in a hyperplane.
Indeed, the set $\hat F(\hat U\cap L_a)$ consists of points $\hat x$
such that $P_{\hat x}$ contains the set $F(U\cap L)$ for some line
$L$ passing through $a$.
But these hyperplanes $P_{\hat x}$ contain the point $F(a)$.
Therefore, all corresponding points $\hat x$ lie in the hyperplane $P_{F(a)}$.
Note that the set $\hat U$ may be empty.
In this case, we say that {\em the dual of $F$ is empty}.
If $\hat F$ is nonempty, and its dual is also nonempty, then
the dual of $\hat F$ coincides with $F$ on some open subset of $U$,
as follows easily from the preceding discussion.

Duality gives rise to more examples of planarizations.
E.g. we can consider dual planarizations to quadratic rational maps $U\to\RP^3$.
These dual planarizations are not, in general, quadratic rational maps.
They are cubic; although not immediately obvious, it can be verified by direct computation.

\subsection{Main results}
The main result of this paper is the following theorem about planarizations
with values in $\RP^3$ (and its applications to maps that take lines to conics):

\begin{thm}
\label{t:2-plan}
  Let $F:U\to\RP^3$ be a planarization.
  Then there exists an open subset $U'\subset U$ such that at least one of the
  following holds:
\begin{itemize}
  \item the map $F|_{U'}$ is a trivial planarization;
  \item the map $F|_{U'}$ is a co-trivial planarization;
  \item the map $F|_{U'}$ is a rational map of degree at most 3.
\end{itemize}
\end{thm}

Not all rational maps $F:\RP^2\dashrightarrow\RP^3$ of degree 3 are planarizations.
However, some of them are (e.g. dual planarizations to generic quadratic
rational maps).
Thus Theorem \ref{t:2-plan} does not give a complete classification of all
planarizations.
It remains to describe cubic rational planarizations.
I think that this is an interesting open problem.

Now consider a linear system $\Lc$ of conics in $\RP^2$.
By definition, any conic from $\Lc$ is given by an equation of the form
$$
\lambda_0\phi_0+\dots+\lambda_n\phi_n=0,
$$
where $\phi_0$, $\dots$, $\phi_n$ are linearly independent real quadratic forms in
homogeneous coordinates of $\RP^2$, and $\lambda_0$, $\dots$, $\lambda_n$
are some real numbers.
The latter are only defined up to a common nonzero factor, thus
they can be thought of as homogeneous coordinates in $\Lc$.
The number $n$ is called the {\em dimension} of $\Lc$.
Linear systems of dimension one are called {\em pencils}, those of
dimension two are called {\em nets}, and those of dimension three
are called {\em webs}.
As an application of Theorem \ref{t:2-plan}, we will prove the following theorem.

\begin{thm}
\label{t:linsys3}
  Consider a linear web $\Lc$ of conics in $\RP^2$.
  Let $f:U\to\RP^2$ be a sufficiently smooth map that takes lines to $\Lc$-curves.
  Then there exists an open subset $U'\subset U$, on which one of the following holds:
\begin{enumerate}
\item
  the set $f(U')$ is a subset of a conic from $\Lc$;
\item
  the map $f|_{U'}$ is a local inverse of a quadratic rational map;
\item
  the map $f|_{U'}$ is a quadratic rational map;
\item
  the map $f|_{U'}$ is a local branch of the multi-valued map 
  $\Phi^{-1}\circ F$, where $\Phi$, $F:\RP^2\dashrightarrow Q$ are quadratic
  rational maps to the same irreducible quadric $Q\subset\RP^3$.
\end{enumerate}
\end{thm}

It is desirable to have a description of all sufficiently smooth maps
that take lines to conics.
An intermediate step would be to describe all maps that take lines
to $\Lc$-curves for linear systems $\Lc$ of conics of dimension 4
(the system of all conics has dimension 5).
Maps that take linear pencils of lines to conics have been studied
in \cite{rectcon}.

\paragraph{\em Organization of the paper.}
In Section 2, we prove Theorem \ref{t:2-plan}.
As an immediate application of this theorem, we give a short proof of
a theorem of Khovanskii \cite{Kh} that describes all sufficiently
smooth maps taking lines to circles.
In Section 3, we prove Theorem \ref{t:linsys3}.

\paragraph{\em Acknowledgements.}
I am grateful to A. Khovanskii for getting me interested in the subject, and to
J. Bernstein, who suggested to replace the system of all circles with
an arbitrary linear system of conics.
I am also grateful to L. Sukhanov for an example that led to a correction
in Theorem \ref{t:linsys3}.

\section{Planarizations}

Suppose that $U\subset\R^2$, and let $F:U\mapsto\RP^n$ be a planarization.
Suppose that $F(U)$ lies in an affine chart of $\RP^n$
so that we can think of $F$ as a map from $U$ to $\R^n$.
Suppose also that $F$ is sufficiently many times differentiable in $U$.
Fix a point $a\in U$.
Let $j^{n-1}_aF:\R^2\to\R^n$ be the Taylor polynomial for $F$ at $a$
of degree $n-1$.
We say that $F$ is {\em nondegenerate} at the point $a$
if for some line $L$ such that $L\ni a$
there is a unique hyperplane containing the set $j^{n-1}_aF(L\cap \R^2)$
(note that the polynomial $j^{n-1}_aF$ is defined on all $\R^2$).
Clearly, in this case the set $F(U\cap L)$ is contained in the same hyperplane.
It is easy to check that the notion of nondegeneracy does not
depend on the choice of an affine chart.
If $F$ is nondegenerate at least at one point, then the dual planarization
$\hat F$ is non-empty.

\begin{prop}
\label{P:cubic}
 Suppose that $F$ is nondegenerate at a point $a\in U$.
 Then the restriction of $\hat F$ to $L_a\cap\hat U$ is a rational map
 of degree at most $\frac{n(n-1)}2$.
\end{prop}

\begin{proof}
  We can assume that $U$ lies in an affine chart and that $a$ is the origin.
  Let $u$, $v$ be affine coordinates on $U$.
  Then $j^{n-1}_aF$ has the form
  $$
  j^{n-1}_aF(u,v)=\sum_{i+j\le n-1} u^iv^j A_{i,j},\quad
  A_{i,j}=\frac{1}{i!j!}\frac{\d^{i+j}F}{\d u^i\,\d v^j}(0,0)
  $$
  by the Taylor formula.
  Consider a line $L=L_{\hat a}$, where $\hat a\in\hat U$.
  Suppose that, in coordinates $(u,v)$, it is given by the equation
  $v=\lambda u$, where $\lambda$ is a constant (the slope of $L$).
  Then the restriction of $j^{n-1}_aF$ to $L$ is parameterized by $u$
  as follows:
  $$
  j^{n-1}_aF(u,\lambda u)=\sum_{i+j\le n-1} \lambda^j u^{i+j} A_{i,j}=
  \sum_{l=0}^{n-1} u^lB_l(\lambda),\quad B_l(\lambda)=\sum_{j=0}^l\lambda^j A_{l-j,j}.
  $$
  Note that $B_l(\lambda)$ is a polynomial of degree $\le l$.
  Consider the embedding $\iota:\R^n\to\R^{n+1}$ given by the formula
  $x\mapsto (x,1)$.
  Then the $(n-1)$-plane $\Pi(\lambda)$ containing the set $j^{n-1}_aF(L\cap\R^2)$
  consists of all points $x\in\R^n$ such that $\iota(x)\wedge\Omega(\lambda)=0$,
  where
  $$
  \Omega(\lambda)=\iota(B_0)\wedge\iota(B_1)\wedge\dots\wedge\iota(B_{n-1})
  $$
  The $n$-vector $\Omega(\lambda)$ depends on $\lambda$ polynomially of degree
  at most $0+1+\dots+(n-1)=\frac{n(n-1)}2$
  and is not identical zero.
\end{proof}

Suppose that $F$ is nondegenerate at least at one point.
Then there is an open set, at every point of which $F$ is nondegenerate.
Therefore, the dual map $\hat F$ becomes rational of degree $\frac{n(n-1)}2$
when restricted to any line in some open set of lines.
We have the following statement about such maps, which generalizes well-known
facts about maps, whose restrictions to lines are polynomials.

\begin{prop}
 \label{p:cubic-on-lines}
Let $V\subset\RP^2$ be an open subset, and $f:V\to\R$ a function, for which there
exists an open set of lines $L\subset\RP^2$ such that $L\cap V\ne\0$ and
$f|_{L\cap V}$ is a rational function of degree $d$.
Then the function $f$ is also a rational function of degree $d$, possibly
after restriction to a smaller open set.
\end{prop}

\begin{proof}
 It suffices to assume that $V\subset\R^2$.
 Let $\Uc$ be the open set of lines such that, for every $L\in\Uc$,
 we have $V\cap L\ne\0$, and $f|_{V\cap L}$ is a rational function
 of degree $d$.
Let $u$ and $v$ be affine coordinates on $\R^2$.
We can assume that $\Uc$ contains horizontal lines $v=const$ passing through
all points of $V$, and vertical lines $u=const$ passing through all points of $V$
(we may need to pass from $V$ and $\Uc$ to smaller open subsets
and make an affine coordinate change to arrange this).
Choose $2d+1$ different real numbers $u_0$, $\dots$, $u_{2d}$.
Note that, if $\phi(u)$ and $\psi(u)$ are degree $\le d$ rational functions of
$u$ such that
$\phi(u_i)=\psi(u_i)$ for all $i=0,\dots,2d$, then $\phi=\psi$.
This follows from the fact that a polynomial of degree at most $2d$ vanishing
at $2d+1$ different points must vanish identically.

Consider a rational function $g_c$ of degree $\le d$ such that
$g_c(u_i)=f(u_i,c)$ for all $i=0,\dots,2d$.
Such a rational function exists and coincides with the function
$u\mapsto f(u,c)$ by the remark just made.
On the other hand, the system of equations $g_c(u_i)=f(u_i,c)$ can be solved
for the coefficients of $g_c$.
Note that, after we multiply both parts of the equation $g_c(u_i)=f(u_i,c)$
by the denominator of $g_c$ evaluated at $u_i$, we obtain a linear equation.
Thus the system of equations $g_c(u_i)=f(u_i,c)$ is equivalent to some overdetermined
linear system.
At least locally, the solution of this system (suitably normalized --- note that
all equations are homogeneous, and the rational function does not change when we
multiply all its coefficients by the same number) is given by a rational function
of $u_i$ and $f(u_i,c)$.
This follows from the Cramer rule.
What is important for us is that the coefficients of
the rational function $g_c$ depend rationally on $c$!
It follows that $f$ is a rational function since $f(u,v)=g_v(u)$.

It remains to prove that the degree of $f$ is at most $d$.
Suppose that $f$ is a rational function of degree $d'>d$ represented as
a ratio of two relatively prime polynomials, at least one of which has degree $d'$.
Then the restrictions of this polynomial to all lines in some Zariski open set of lines
are polynomials of degree $d'$.
Moreover, the restrictions of two relatively prime polynomials to any
line in some Zariski open set of lines are relatively prime.
Thus the restrictions of $f$ to all lines in some Zariski open set of lines
will be rational functions of degree $d'$, a contradiction.
\end{proof}

Propositions \ref{P:cubic} and \ref{p:cubic-on-lines} imply the following

\begin{thm}
\label{T:dual-is-rational}
  Suppose that $F:U\to\RP^n$ is a planarization
  that is nondegenerate at some point.
  Then the dual planarization $\hat F:\hat U\to\RP^{n*}$ is
  rational of degree at most $\frac{n(n-1)}2$
  (possibly after restriction to some open subset of $\hat U$).
\end{thm}

We now discuss the case $n=3$.

\begin{prop}
\label{P:degen}
  Let $U\subset\R^2$ be an open subset.
  Consider a planarization $F:U\to\R^3$.
  If there is no point of $U$, at which $F$ is nondegenerate,
  then the planarization $F$ is trivial, possibly after restriction to a
  smaller open set.
\end{prop}

\begin{proof}
If the differential $dF$ of $F$ vanishes everywhere on $U$, then
$F$ is locally constant on $U$,
in particular, the restriction of $F$ to any connected component
of $U$ is a trivial planarization.
Suppose that $dF\ne 0$ at least at one point.
Then there is an open set $U'\subset U$ of points and an open set
$V'$ of vectors in $\R^2$ such that $d_a F(\xi)\ne 0$ for every $a\in U'$
and $\xi\in V'$.

Since the map $F$ is degenerate at all points of $U'$, we have
$d_aF(\xi)\wedge d^2_aF(\xi,\xi)=0$ for all points $a\in U'$ and all vectors
$\xi\in V'$.
For $a\in\R^2$ and $\xi\in V'$, we set $X(t)=F(a+t\xi)$.
We have $\dot X(t)\ne 0$ whenever $a+t\xi\in U'$.
On the other hand, $\ddot X\wedge \dot X=0$.
It follows that the $X$-image of any interval $(t_1,t_2)$ such
that $a+t\xi\in U'$ for all $t\in (t_1,t_2)$ lies in a line.
Hence, if $U''$ is any convex open subset of $U'$, then $F(U''\cap L)$
lies in a line for every line $L$ of the form $\{a+t\xi \,|\, t\in\R\}$,
where $\xi\in V'$.
The set of such lines is an open set.
It now follows by the M\"obius theorem that, for
some open subset $U'''\subset U$, either the set $F(U''')$
is a subset of a line, or $F|_{U'''}$ is a projective embedding.
In both cases, $F$ is trivial on $U'''$.
\end{proof}

\begin{proof}[Proof of Theorem \ref{t:2-plan}]
  If there is no point, at which $F$ is nondegenerate, then
  by Proposition \ref{P:degen}, there is an open subset of $U$, on
  which $F$ is a trivial planarization.
  Suppose now that $F$ is nondegenerate somewhere.
  Then the dual planarization $\hat F$ is defined on some open subset
  $\hat U\subset\RP^{2*}$.
  If there is no point, at which $\hat F$ is nondegenerate,
  then, again by Proposition \ref{P:degen}, there is an open subset
  of $\hat U$, on which $\hat F$ is a trivial planarization.
  It follows that $F$ is a co-trivial planarization (on an open
  subset of $U$).
  Finally, suppose that $\hat F$ is nondegenerate at some point.
  Then the restriction of $F$
  to some open subset of $U$ coincides with the dual of $\hat F$,
  hence, by Theorem \ref{T:dual-is-rational}, it is a rational map of degree at most three.
\end{proof}

Let us now discuss some immediate applications of Theorem \ref{t:2-plan}.
The following result is essentially equivalent to the theorem of Khovanskii
\cite{Kh}:

\begin{thm}
\label{t:khov}
  Let $U$ be an open subset of $\RP^2$, and $f:U\to S^2$ a sufficiently
  smooth map such that the image $f(U\cap L)$ lies in a circle for
  every line $L\subset\RP^2$.
  Then, possibly after restriction to a smaller open set,
  one of the following holds:
\begin{enumerate}
\item
the image of $f$ lies in a circle;
\item
the map $f$ is a co-trivial planarization;
\item
the map $f$ is a quadratic rational map.
\end{enumerate}
\end{thm}

Note that circles in $S^2$ are precisely intersections of $S^2$ with
2-dimensional planes.
Using the methods developed above, we give a new proof of this theorem.
The following lemmas are left as (simple) exercises:

\begin{lem}
  Suppose that $P$, $Q$ and $R$ are cubic polynomials in one variable $t$
  with complex coefficients such that $PQ=R^2$.
  Then $P$, $Q$ and $R$ have a nonconstant common divisor in $\C[t]$.
\end{lem}

\begin{lem}
\label{l:cub-qua}
  Any rational map of $\CP^1$ to $\CP^n$
  of degree at most three,
  whose image lies in an irreducible conic,
  is in fact a quadratic rational map.
\end{lem}

Note that Lemma \ref{l:cub-qua} applies also to rational maps
of $\RP^1$ to $\RP^n$ since any such map can be complexified.

\begin{lem}
  \label{l:cub-qua1}
  Consider a rational map $F:\CP^2\dashrightarrow\CP^n$ of degree at most three
  such that, for a Zariski dense set of lines $L\subset\CP^2$,
  the image $F(L)$ lies in a plane algebraic curve of degree at most two.
  Then $F$ is a rational map of degree at most two, unless $F(\CP^2)$
  is contained in a line.
\end{lem}

\begin{proof}
 Suppose that, for at least one line $L\subset\CP^2$ (hence for
 a Zariski open and dense set of lines), the set $F(L)$ lies in an irreducible conic.
 The map $F$ restricted to any such line must be a rational map of
 degree at most 2 by Lemma \ref{l:cub-qua}.
 Hence by Proposition \ref{p:cubic-on-lines} (more precisely, an obvious
 complex analog of it and Theorem \ref{T:dual-is-rational}), the map $F$
 is a rational map of degree at most two.
 If, for every $L$, the set $F(L)$ lies in a line, then by the M\"obius theorem,
 $F$ is a projective transformation, i.e. a rational map of degree 1 (unless
 $F(\CP^2)$ is contained in a line).
\end{proof}

\begin{proof}[Proof of Theorem \ref{t:khov}]
If the planarization $f$ is trivial, then we have case $(1)$ of the theorem.
Suppose that $f$ is not trivial.
By Theorem \ref{t:2-plan}, it can be either co-trivial or rational of degree at most 3.
In the latter case, by Lemma \ref{l:cub-qua1},
the map $f$ is a rational map of degree at most two.
If $f$ were a projective map, then $f$ would be a trivial planarization.
Therefore, $f$ is a quadratic rational map.
\end{proof}

\section{Linear webs of conics}

Consider a linear system $\Lc$ of conics in $\RP^2$.
The linear system $\Lc$ defines a quadratic rational map $\Phi_\Lc:\RP^2\to\RP^n$,
where $n$ is the dimension of $\Lc$.
Fix a system $[\lambda_0:\dots:\lambda_n]$ of homogeneous coordinates in $\Lc$,
then the conic with homogeneous coordinates $[\lambda_0:\dots:\lambda_n]$
is given by the equation
$$
\lambda_0\phi_0+\dots+\lambda_n\phi_n=0,
$$
where $\phi_0$, $\dots$, $\phi_n$ are fixed
linearly independent homogeneous quadratic polynomials in the
homogeneous coordinates on $\RP^2$.
The map $\Phi_\Lc$ is defined by the equation
$$
\Phi_\Lc[x_0:x_1:x_2]=[\phi_0[x_0:x_1:x_2]:\dots:\phi_n[x_0:x_1:x_2]].
$$

Consider first the case $n=2$, i.e. a linear net of conics.

\begin{prop}
\label{p:linsys2}
  Suppose that $U\subset\RP^2$ is an open set, and $f:U\to\RP^2$ is
  a sufficiently smooth map that takes lines to $\Lc$-curves for a linear
  net $\Lc$ of conics.
  Then either the image of $f$ lies in a conic from $\Lc$,
  or $f$ is a local inverse of some quadratic rational map
  (namely, of the map $\Phi_\Lc$ post-composed with some M\"obius
  transformation), possibly after restriction to a smaller open set.
\end{prop}

\begin{proof}
The map $\Phi_\Lc\circ f:U\to\RP^2$ takes lines to lines.
Therefore, by the theorem of M\"obius, the map $\Phi_\Lc\circ f$
is a restriction of a projective transformation $P:\RP^2\to\RP^2$
(unless the image of $\Phi_\Lc\circ f$ lies in a line, i.e. the
image of $f$ lies in a conic).
It follows that $f$ is a local branch of
$\Phi_\Lc^{-1}\circ P=(P^{-1}\circ\Phi_{\Lc})^{-1}$.
\end{proof}

We now proceed with the proof of Theorem \ref{t:linsys3}.
Let $\Lc$ be a linear web of conics defined by linearly
independent homogeneous quadratic polynomials $\phi_0$,
$\phi_1$, $\phi_2$ and $\phi_3$ in the homogeneous coordinates of $\RP^2$.
The map $f$ gives rise to a planarization $F=\Phi_\Lc\circ f:U\to\RP^3$.

Let $S$ denote the Zariski closure in $\CP^3$ of the set $\Phi_\Lc(\CP^2)$.
This is an irreducible algebraic surface of degree at most 4
(if $S$ were a curve, then it would be a conic, but we know that $S$
is not a subset of a plane).
The estimate on the degree follows from the elimination theory.
If the degree of $S$ is two, then we have case $(4)$ of Theorem \ref{t:linsys3}, since
$F$ must be a quadratic rational map by Lemma \ref{l:cub-qua1},
and $f=\Phi_\Lc^{-1}\circ F$.

\begin{lem}
  \label{l:Phi-birat}
  Suppose that the degree of $S$ is at least 3.
  Then the map $\Phi_\Lc$ is a bi-rational isomorphism between $\RP^2$ and
  $\Phi_\Lc(\RP^2)$.
\end{lem}

\begin{proof}
It suffices to prove that $\Phi_\Lc:\CP^2\to S$ is a bi-rational isomorphism,
since $\Phi_\Lc^{-1}$ will then be automatically defined over real numbers.

  Take a generic line $L\subset\CP^2$.
  The image $\Phi_\Lc(L)$ is an irreducible conic $C$ (otherwise $\Phi_\Lc$
  takes lines to lines, hence by the M\"obius theorem $\Phi_\Lc(\RP^2)$
  is a planar set, a contradiction with the assumption that $\phi_i$
  are linearly independent).
  Let $P_L$ be the 2-dimensional plane containing $C$.
  If $P_L\cap S=C$ set-theoretically, then $C$ consists of tangency points between $S$ and $P_L$.
  There cannot be infinitely many such conics $C$, hence we can assume that
  $P_L\cap S$ is strictly bigger than $C$.
  Since $\Phi_\Lc^{-1}(P_L)$ is a conic containing $L$, this conic
  is the union of $L$ and a line $L'\ne L$.
  The restriction of $\Phi_\Lc$ to each of the two lines $L$, $L'$
  has degree one, and the images of these two lines are different.
  Hence the degree of $\Phi_{\Lc}$ is one.
\end{proof}

\begin{lem}
  \label{l:rat}
  Suppose that a rational map $f:\RP^2\dashrightarrow\RP^2$ takes lines to conics.
  Then $f$ has degree at most two, unless $f(\RP^2)$ is a subset of a conic.
\end{lem}

\begin{proof}
  Replace $f$ with its complexification $f:\CP^2\dashrightarrow\CP^2$.
  It also takes lines to conics.
  Suppose that $f$ is a rational map of degree $>2$.
  Then the restriction of $f$ to a generic line is many-to-one.
  This contradicts the fact that a generic point has finitely many
  preimages under $f$.
\end{proof}

\begin{proof}[Proof of Theorem \ref{t:linsys3}]
 By Theorem \ref{t:2-plan}, there exists an open subset $U'\subset U$
 such that the map $F|_{U'}$ is trivial, or co-trivial, or rational
 of degree at most 3.
 If $F|_{U'}$ is a trivial planarization, then $f(U')$ is a subset of a
 conic from $\Lc$, i.e. case $(1)$ of the theorem takes place.
 If $F|_{U'}$
 is a co-trivial planarization, then, by definition of $F$,
 there is a linear net $\Lc'\subset\Lc$ such that $f$ takes all
 lines to $\Lc'$-curves.
 In this case, by Proposition \ref{p:linsys2}, the map $f$ is a local
 inverse of some quadratic rational map, i.e. case $(2)$ or the
 theorem takes place.

 Suppose now that $F|_{U'}$ is a rational map.
 If the degree of $S$ is two, then case $(4)$ of the theorem takes place.
 Otherwise, since $\Phi_\Lc^{-1}:\Phi_\Lc(\RP^2)\dashrightarrow \RP^2$ is a bi-rational map
 by Lemma \ref{l:Phi-birat},
 the composition $f|_{U'}=\Phi_{\Lc}^{-1}\circ F|_{U'}$ is a
 rational map.
 Since it takes lines to conics, it follows from Lemma \ref{l:rat}
 that $f|_{U'}$ is in fact a quadratic rational map, i.e. case $(3)$
 of the theorem takes place.
\end{proof}


\begin{thebibliography}{9999}

\bibitem[GKh]{GKh}
G.S. Khovanskii, {\em Foundations of Nomography}, ``Nauka'',
Moscow, 1976 (Russian)

\bibitem[Iz]{Iz}
F.A. Izadi
{\em On rectification of circles and an extension of Beltrami's theorem},
Rocky mountain J. of Math. Vol. \textbf{34} (2005), No. 3

\bibitem[Kh]{Kh}
A.G. Khovanskii, {\em Rectification of circles}, Sib. Mat. Zh.,
{\bf 21} (1980), 221--226

\bibitem[M]{M}
A.F. M\"obius, {\em Der barycentrische Calcul, 1827}, in:
August Ferdinand M\"obius, gesammelte Werke, Band \textbf{1} -S. Hirzel (Ed.),
Leipzig 1885--1887

\bibitem[vS]{vS}
K. G. Ch. von Staudt, {\em Geometrie der Lage}, N\"urnberg 1847

\bibitem[T03]{rect}
V. Timorin,
{\em Rectification of circles and quaternions},
Michigan Mathematical Journal, \textbf{51} (2003), 153--167

\bibitem[T04]{cliff}
V. Timorin,
{\em Circles and Clifford algebras}, Functional Analysis and its
Applications, \textbf{38} (2004), No. 1, 45--51,

\bibitem[T05]{circquad}
V. Timorin,
{\em Circles and quadratic maps between spheres},
Geometriae Dedicata \textbf{115} (2005), pp. 19--32,

\bibitem[T06]{maplc}
V. Timorin,
{\em Maps That Take Lines To Circles, in Dimension 4},
Functional Analysis and its Applications \textbf{40} (2006), no. 2, 108--116

\bibitem[T07]{rectcon}
V. Timorin,
{\em Rectifiable pencils of conics}, Moscow Mathematical Journal
\textbf{7} (2007), no. 3, 561--570
\end{thebibliography}
\end{document}